\documentclass[11pt, reqno]{amsart}
\usepackage{version}
\usepackage{amssymb}
\usepackage[]{enumitem}
\usepackage{enumitem}

\usepackage{hyperref}
\usepackage{lipsum}

\newcommand{\myfootnote}[1]{
    \renewcommand{\thefootnote}{}
    \footnotetext{\hspace{-2pt}\scriptsize#1}
    \renewcommand{\thefootnote}{\arabic{footnote}}
}

\newcommand*{\doi}[1]{\href{http://dx.doi.org/#1}{DOI: #1}}

\theoremstyle{definition}
\newtheorem{thm}{Theorem}[section]

\newtheorem{lem}{Lemma}[section]

\newtheorem{rk}{Remark}[section]
\newtheorem*{akn}{Acknowledgments}

\usepackage{dirtytalk}
\usepackage{mathrsfs}
\usepackage{amsmath, amsfonts, amssymb,amsthm}
\newcommand{\field}[1]{\mathbb{#1}}
\newcommand{\N}{\field{N}}

\newcommand{\R}{\field{R}}
\newcommand{\C}{\field{C}}
\newcommand{\B}{\field{B}}

\DeclareMathOperator{\imm}{Im}
\DeclareMathOperator{\id}{Id}

\DeclareMathOperator{\spann}{Span}

\numberwithin{equation}{section}

\begin{document}
\title[]{Approximation and accumulation results of holomorphic mappings with dense image}
\author{Giovanni D. Di Salvo}
\address{Department of Mathematics\\
University of Oslo, Postboks 1337 Blindern\\
0316 Oslo\\
Norway}
\email{giovannidomenico.disalvo@gmail.com}

%
%
\myfootnote
{
    Published in \emph{Mathematica Scandinavica},
    2023,
    volume~129,
    issue~2,
    pp.~273--283.
    }    
\myfootnote
{
	\doi{10.7146/math.scand.a-136450}.
}
\begin{abstract}
We present four approximation theorems for manifold--valued mappings. The first one approximates holomorphic embeddings on pseudoconvex domains in $\C^n$ with holomorphic embeddings with dense images.
The second theorem approximates holomorphic mappings on complex manifolds with bounded images with holomorphic mappings with dense images.
The last two theorems work the other way around, constructing (in different settings) sequences of holomorphic mappings (embeddings in the first one) converging to a mapping with dense image defined on a given compact minus certain points (thus in general not holomorphic).
\end{abstract}

\maketitle
\section{Introduction}

Let $\triangle$ denote the unit disc in $\C$, $Y$ be a connected complex manifold, and $\mathcal O(\triangle,Y)$ denote the set of all analytic discs in $Y$ , i.e., all holomorphic maps $f\colon\triangle\to Y$ from $\triangle$ into $Y$.
The motivation for this study is a result of F. Forstneri\v c and J. Winkelmann \cite{fourthpaper:F17} which states
that the set $\mathcal D\subseteq \mathcal O(\triangle,Y)$ consisting of all analytic discs with dense image is dense in
$\mathcal O(\triangle,Y)$ with respect to the compact convergence topology.

Another result accounting this topic can be found in \cite{fourthpaper:AF18}: given a closed complex submanifold $X\subsetneq\C^n$, for $n>1$, there exists a \emph{complete} (the image of every divergent path in $X$ has infinite length in $\C^n$) holomorphic embedding $f\colon X\hookrightarrow \C^n$ with everywhere dense image; for $n=1$ the same result holds for complete holomorphic embeddings $f\colon\C\hookrightarrow\C^2$.
If moreover $X\cap\B^n\neq0$ and $K\subset X\cap\B^n$ is compact, there exists a Runge domain $\Omega\subset X$ containing $K$ which admits a complete holomorphic embedding $f\colon \Omega\hookrightarrow \B^n$ with dense image.

Finally in \cite{fourthpaper:FS77} it is proved the existence of a holomorphic injective mapping with dense image from the open unit polydisc in $\C^m$ to an $m$--dimensional paracompact connected complex manifold $M$.\\
\newline

The purpose of this paper is to generalize the theorem of Forstneri\v c and Winkelman
in four different ways. Before we can state our results we need to introduce some notation
and definitions.\\
\newline

Let $X$ and $Y$ complex connected manifolds and let $\mathcal O(X,Y)$ denote the set of all holomorphic maps $h\colon X\to Y$. We equip $\mathcal O(X,Y)$ with the \emph{compact--open} topology. We are interested in the density of certain subclasses of $\mathcal O(X,Y)$ with respect to this topology, so we assume that we have metrics 
$d_X$ and $d_Y$ defining them.\\
\newline

In the case when $\dim X \le \dim Y$ we let $\mathcal O_{\operatorname{e}}(X,Y)$  denote the set of all holomorphic
embeddings of $X$ into $Y$, i.e., mappings that are homeomorphic onto the image. We let $\mathscr H(X,Y)$ denote the set of all holomorphic maps with dense image in $Y$,
$$
\mathscr H(X,Y):=\{h\in\mathcal O(X,Y)\;:\;\overline{h(X)}=Y\}
$$
and $\mathscr G(X,Y)$ denote the set of all holomorphic maps that are non--constant and have relatively compact image in $Y$,
$$
\mathscr G(X,Y):=\{g\in\mathcal O(X,Y)\;:\;g\;\mbox{is non--constant},\;g(X)\subset\subset Y\}\:.
$$
Observe that $\mathscr G(X,Y)=\emptyset$ whenever $X$ is either compact or euclidean. In both cases all holomorphic maps are constant, for in the euclidean case Liouville's theorem states that bounded maps are constant. We have $\mathscr G(X,Y)\neq\emptyset$ for every relatively compact subdomain $X$ of a Stein manifold $Y$.\\
\newline

We say that $\mathscr H(X,Y)$ \emph{compactly approximates} $\mathscr G(X,Y)$ if for every $g\in\mathscr G(X,Y)$, every compact subset $M$ of $X$ and every $\epsilon>0$, there exists $h\in\mathscr H(X,Y)$ such that
$$
\sup_{x\in M}d_Y(g(x),h(x))<\epsilon\;.
$$

We say that a point $\zeta$ in a compact subset $K$ of a complex manifold $Y$ is a \emph{locally exposable point} in $K$ if there exists a $\mathscr{C}^2$--smooth strictly plurisubharmonic function $\rho$ on some open neighborhood $U$ of $\zeta$ such that
\begin{enumerate}
\item $\rho(\zeta)=0$ and $d\rho(\zeta)\neq0$\;, and
\item $\rho<0$ on $(K\cap U)\setminus\{\zeta\}$\;.
\end{enumerate}

This concept was first introduced in \cite{fourthpaper:DFW18} and it generalizes the more standard concept of \emph{local peak point}, see \cite{fourthpaper:Ran86}, pag. 354. Now we can state all our main results.

\begin{thm}\label{fourthpaper:thm1}
If $X\subset\C^n$ is a pseudoconvex domain which is bounded and star--shaped with respect to $0$, and $n\le\dim Y$, then the set of all $h\in \mathcal O_{\operatorname{e}}(X,Y)$ with dense image in $Y$ is a dense subset of $\mathcal O_{\operatorname{e}}(X,Y)$.   
\end{thm}

\begin{thm}\label{fourthpaper:thm2}
If $Y$ is Stein and $X$ is any connected manifold, then $\mathscr H(X,Y)$ compactly approximates $\mathscr G(X,Y)$.
\end{thm}

\begin{thm}\label{fourthpaper:thm3}
Let $K$ be a connected Stein compact in $Y$, which is not finite.  Then the set of locally exposable points in $K$ is non--empty and for every such point $x_0\in K$, there exist an open neighborhood $U_k$ of $K$ and an injective map $F_k\in\mathcal O(U_k,Y)$ which converges as $k\to+\infty$ to some $F\colon K\setminus\{x_0\}\to Y$ such that $\overline{F(K\setminus\{x_0\})}=Y$.
\end{thm}

\begin{thm}\label{fourthpaper:thm4}
Let $Y$ be a Stein manifold, $K\subset Y$ connected not finite compact.\\
Denote with $\Gamma$ the closure in $Y$ of the set of locally exposable points for $K$, which is non--empty. Then there exist $U_k\subset Y$ open neighborhood of $K$ and $F_k\in\mathcal O(U_k,Y)$ converges as $k\to+\infty$ to some $F\colon K\setminus\Gamma\to Y$ such that $\overline{F(K\setminus\Gamma)}=Y$.
\end{thm}

\akn I warmly thank E. F. Wold for directing me to write this paper.

\medskip

\section{Technical Tools}

The proofs will extensively exploit Theorem 1.1 in \cite{fourthpaper:DFW18} and a slightly different version of it (whose proof follows automatically from the original one) which is as follows:

\begin{thm}\label{fourthpaper:Erlend}
Let $Y$ be a complex manifold and $Y_0\subset Y$ Stein compact.
Let $\zeta\in Y_0$ be locally exposable and $\gamma\colon[0,1]\to Y$ smoothly embedded curve such that $\gamma(0)=\zeta$ and $\gamma((0,\delta])\subset Y\setminus Y_0$ for some $\delta>0$.
Then, for every $V$ open neighborhood of $\gamma$ and for every $\epsilon>0$, there exist the following:
\begin{enumerate}
	\item $U\subset Y$ neighborhood of $Y_0$\;, 
	\item an arbitrarily small $V'\subset V$ neighborhood of $\zeta$\;, and
	\item $f\colon U\to Y$ holomorphic mapping such that
		\begin{itemize}
		\item $f(\zeta)=\gamma(1)$\;;
		\item $f(V')\subset V$\;;
		\item $\|f-\id\|_{Y_0\setminus V'}<\epsilon$\;.
		\end{itemize}
\end{enumerate}
If the whole curve $\gamma((0,1])$ lies in $Y\setminus Y_0$ then $f$ can be chosen to be injective on $U$.
\end{thm}
In \cite{fourthpaper:DFW18} the theorem is formulated in a slightly more general setting and considering only the case $\gamma((0,1])\subset Y\setminus Y_0$. 

\begin{lem}\label{fourthpaper:tech}
Let $X$ be a metrizable topological space, $W\subset\subset X$ connected not finite. Let $f_k\colon U_k\to X$ be a sequence of continuous mappings, where $U_k\subset X$ is some open neighborhood of $K_k$ and $K_1:=\overline W$, $K_{k+1}:=f_k(K_k)$.
Define $F_k:=f_k\circ\cdots\circ f_1\colon\overline W\to X$ and consider
\begin{itemize}
\item $\{\epsilon_k\}_k$ positive real numbers\;,
\item $\{C_k\}_k$ \emph{compact exhaustion} of $W$, that is $C_k\subset X$ compact, $C_k\subsetneq C_{k+1}$ and $\bigcup_k C_k=W$, and
\item $\{V_{k}'\}_k$ open sets in $X$\;,
\end{itemize}
such that, setting $r_{k}:=\max\{d(x,F_{k}(C_{k}))\;:\;x\in K_{k+1}\},$ the following hold:
\begin{enumerate}[label=(\roman*)$_k$]
\item $V_k'\subset U_k$ and $V_k'\cap K_k\neq\emptyset$\;,
\item \label{fourthpaper:max} $\max\{r_k,\;\|f_k-\id\|_{K_k\setminus V_k'}\}\le\epsilon_k$,\;and
\item \label{fourthpaper:cap} $V_{k+1}'\cap F_{k}(C_{k})=\emptyset$\;.
\end{enumerate}
If $F_k$ converges uniformly on compacts of $W$ to $F\colon W\to X$, then for every $x\in K_{k+1}$
\begin{equation}\label{fourthpaper:dense}
d(x,F(W))\le\sum_{n\ge k}\epsilon_{n}\;.
\end{equation}
\end{lem}

\begin{proof}
Let $x\in K_{k+1}$. Then by definition and from \ref{fourthpaper:max} one gets
$$
d(x,F_k(C_k))\le r_{k}\le\epsilon_{k}\;.
$$
\newline
Then \ref{fourthpaper:cap} implies $F_k(C_k)\subset K_{k+1}\setminus V_{k+1}'$ hence \ref{fourthpaper:max}$_{+1}$ says that $f_{k+1}$ moves $F_k(C_k)$ less than $\epsilon_{k+1}$, therefore 
$$
d(x,f_{k+1}(F_k(C_k)))\le\epsilon_k+\epsilon_{k+1}\;.
$$
Since $f_{k+1}(F_k(C_k))=F_{k+1}(C_k)\subset F_{k+1}(C_{k+1})\subset K_{k+2}\setminus V_{k+2}'$ we can repeat the argument getting
$$
d(x,\underbrace{f_{k+2}f_{k+1}(F_k(C_k))}_{F_{k+2}(C_k)})\le\epsilon_k+\epsilon_{k+1}+\epsilon_{k+2}\;.
$$
Inductively, for every $m\ge0$ we get
$$
d(x,F_{k+m}(C_k))\le\sum_{j=0}^m\epsilon_{k+j}
$$
\newline
and passing to $\lim_{m\to\infty}$ (which is well defined in left hand side, since the distance is continuous and $\{F_n\}_n$ uniformly converges to $F$ on $C_k$) we get $d(x,F(C_k))\le\sum_{j\ge0}\epsilon_{k+j}$.
Since $F(C_k)\subset F(W)$, we have $d(x,F(W))\le d(x,F(C_k))$.
\end{proof}
Recall now a useful property of Stein manifolds (\cite{fourthpaper:Ran86}, \S 2, Proposition 2.21 and Theorem 3.24):
\begin{thm}\label{fourthpaper:morse}
Let $Y$ be a Stein manifold. Then there exists $\rho$, a $\mathscr C^2$--smooth strictly plurisubharmonic exhausting function for $Y$, such that the set of critical points $C:=\{z\in Y\;:\;d\rho(z)=0\}$ is discrete in $Y$.
In particular, for every $c\in\R$, $\{\rho<c\}\subset\subset Y$ and $Y_c:=\{\rho\le c\}$ is a Stein compact. 
\end{thm}
\begin{rk}\label{fourthpaper:boundaryptslocexp}
With the notation of Theorem \ref{fourthpaper:morse}, we see that every regular boundary point of a strictly pseudoconvex domain $\{\rho<c\}$ is locally exposable: take $\zeta\in\{\rho=c\}\setminus C$; we assume it is the origin in suitable local coordinates. Then considering $\tilde\rho(z):=\rho(z)-c-\epsilon|z|^2$ for $\epsilon$ small enough, we conclude.
\end{rk}


\section{Inductive procedure}
Consider $Q=\{q_n\}_n\subset Y$ such that $\overline Q=Y$.
Fix $\epsilon>0$ and define $\epsilon_k:=\frac\epsilon{2^{k+1}}$.
For Theorems \ref{fourthpaper:thm1} and \ref{fourthpaper:thm2} we fix a compact subset $M$ of the domain $X$.
In what follows, we exploit Theorem \ref{fourthpaper:Erlend} to get a suitable sequence of holomorphic functions $f_k$ allowing to reach all the points of $Q$ which have not been already reached, so that the image of the composition tends to invade the whole codomain. Along with the $f_k$, we construct both a compact exhaustion $\{C_k\}_k$ of the case domain $X$, $K\setminus\{x_0\}$ or $K\setminus\Gamma$ so that it fulfills the hypothesis of Lemma \ref{fourthpaper:tech}, and all the other sequences needed.
For the construction of the $f_k$, we will focus on Theorems \ref{fourthpaper:thm2} and \ref{fourthpaper:thm1}; the procedure for the remaining two theorems is alike and it is afterwards explained.
Similarly for the proofs: we worked out the details of the proof of Theorem \ref{fourthpaper:thm2}, which is displayed by proving convergence of the composition of the $f_k$, approximation of the given function, and density of the image. The argument for the remaining three results is analogous and it is subsequently illustrated.
\subsection{Construction of $f_k$ for Theorem \ref{fourthpaper:thm2}}\label{fourthpaper:indthm2}
\subsubsection{Existence of locally exposable points and base of the induction}\label{fourthpaper:elep2}
Consider $g\colon X\to Y$ holomorphic non--constant such that $g(X)\subset\subset Y$. Set $K_1:=\overline{g(X)}$ and take $c_1\in\R$ such that $K_1\subseteq Y_{c_1}$ and $K_1\cap\partial Y_{c_1}\neq\emptyset$ (see notation of Theorem \ref{fourthpaper:morse}) . Consider $\zeta_1\in K_1\cap\partial Y_{c_1}$; if $\zeta_1\notin C$ then it is locally exposable by Remark \ref{fourthpaper:boundaryptslocexp}, otherwise we slightly move $K_1$ by composing $g$ with a suitable holomorphic small perturbation defined as follows.
Assume $\zeta_1$ is the only point in $K_1\cap\partial Y_{c_1}$. Observe that in suitable local coordinates on $\C^n$, which we split as $z=(z',z'')=x+iy=(x'+iy',x''+iy'')\in\C^k\oplus\C^{n-k}$, we can express $\rho$ near the origin (\cite{fourthpaper:F17}, Lemma 3.10.1) as
$$
\rho(z)={\rho(0)}+E(x',x'',y',y'')+o(|z|^2)
$$
where
$$
E(x',x'',y',y''):=\sum_{j=1}^k(\lambda_jy_j^2-x_j^2)+\sum_{j=k+1}^n(\lambda_jy_j^2+x_j^2)
$$
with $\lambda_j>1$ for $j=1,\dots,k$ and $\lambda_j\ge1$ for $j=k+1,\dots,n$, for some $k\in\{0,1,\dots,n\}$.
Assume $\zeta_1$ to be the origin $0\in\C^n$ in these coordinates, so $\rho(0)=\rho(\zeta_1)=c_1$.
The boundary of any $Y_c$ is defined by $\rho$, so we want to move a small neighborhood of $\zeta_1$ in a suitable direction allowing it to go across the level set $\partial Y_{c_1}$.
Take then any nonzero vector $v=(\xi'+i\eta',\xi''+i\eta'')$ with $\xi'=0\in\R^k$.
By standard results on Stein Manifolds (\cite{fourthpaper:Hor90}, Corollary 5.6.3), given $w\in T_{\zeta_1}Y$ there exists a holomorphic vector field $V\colon Y\to TY$ such that $V(\zeta_1)=w$. In our case we take $w$ to correspond to $v$. The flow of $V$ on some neighborhood $W$ of $K_1$ is a holomorphic mapping $\phi_t\colon W\to Y$ defined for complex times sufficiently small in modulus, say $|t|<T\;,t\in\C$. In local coordinates around $\zeta_1$ it is $\phi_t(0)=tv+o(|t|^2)$ and up to shrinking $T$, we have that
$$
\rho(\phi_t(0))=\rho(0)+t^2E(0,\xi'',\eta',\eta'')+o(|t|^3)>\rho(0)=c_1
$$
for any $|t|<T$, considering now $t\in\R$.
By continuity of $(z,t)\mapsto\phi_t(z)$ and of $\rho$, the flow will drag a whole small neighborhood of $\zeta_1$ beyond $\partial Y_{c_1}$ and for sufficiently small times it will be the only piece of $K_1$ crossing the boundary (as we are assuming that $K_1\cap\partial Y_{c_1}$ is just one point), that is: there exist $T'<T$ and $U$ neighborhood of $\zeta_1$ sufficiently small, so that
\begin{itemize}
\item $\rho(\phi_t(z))>\rho(0),\;\;\forall z\in U\cap K_1,\;\;T'/2<t<T'$\;;
\item $\rho(\phi_t(z))\le\rho(0),\;\;\forall z\in K_1\setminus U,\;\;\;0\le t<T'$\;;
\item $\phi_t(U\cap K_1)\cap C=\emptyset,\;\;\forall\; T'/2<t<T'$.
\end{itemize}
So, considering now $\widetilde K_1:=\phi_t(K_1)$ for any $T'/2<t<T'$, there exists some $c>c_1$ such that $\widetilde K_1\subseteq Y_c$ and $\widetilde K_1\cap\partial Y_c\setminus C\neq\emptyset$.

If $K_1\cap\partial Y_{c_1}$ contains more than one point, either it is not discrete or it is a finite set. In the former case $K_1\cap\partial Y_{c_1}\setminus C\neq \emptyset$, in the latter we may assume $K_1\cap\partial Y_{c_1}=\{\zeta_1,\zeta_1'\}$; then, applying the previous argument to one of these points, the piece of $K_1$ that is dragged across $\partial Y_{c_1}$, could be not only $U\cap K_1$, but also $U'\cap K_1$, for any small time (where $U,U'\subset Y$ are suitably small neighborhoods of $\zeta_1$ and $\zeta_1'$ respectively), leading to a similar $\widetilde K_1$ and achieving the same conclusion. Therefore, up to consider $\phi_t\circ g$ instead of $g$, we may assume that $\exists \zeta_1\in K_1\cap\partial Y_{c_1}\setminus C$, which is then locally exposable by Remark \ref{fourthpaper:boundaryptslocexp} and could be sent to any point of $Y\setminus K_1$ by the holomorphic mapping $f_1$ provided by Theorem \ref{fourthpaper:Erlend} (see previous section).

Finally define $f_0:=\id_Y$, $C_0:=\emptyset$, $F_0:=f_0$ and $n_0:=0$.

\subsubsection{Inductive step}\label{fourthpaper:indpr2}
Assume we have the following: $K_k\subset Y$ compact, $c_k\in\R$ such that $K_k\subseteq Y_{c_k}$ and $\exists\zeta_k\in K_k\cap\partial Y_{c_k}\setminus C$, $F_{k-1}$ holomorphic on some neighborhood of $K_1$, with $F_{k-1}(K_1)=K_k$, $C_{k-1}\subset X$ compact and $n_{k-1}\in \N$.
Consider then a smoothly embedded curve $\gamma_k\colon[0,1]\to Y$ such that
\begin{enumerate}[label=(\roman*)$_k$]
\item $\gamma_k(0)=\zeta_k$\;;
\item $\gamma_k(1)=q_{n_k}$ where $n_k:=\min\{n>n_{k-1}\;:\;q_n\notin K_k\}$\;;
\item $\gamma_k((0,\delta_k])\subset Y\setminus Y_{c_k}$ for some $\delta_k>0$\;.
\end{enumerate}
Then for every $V_k$ open neighborhood of $\gamma_k$, Theorem \ref{fourthpaper:Erlend} guarantees there exist
\begin{enumerate}[label=(\arabic*)$_k$]
\item
 $U_k\subset Y$ neighborhood of $Y_{c_k}$\;;
\item
$V_{k}'\subset (V_{k}\cap B(\zeta_{k},\epsilon_{k}))\setminus F_{k-1}(g(M\cup C_{k-1}))$\;;\label{fourthpaper:ind}
\item 
$f_k\colon U_k\to Y$ holomorphic such that the following hold:
\begin{enumerate}[label=(\alph*)$_k$]
\item
$f_k(\zeta_k)=q_{n_k}$\;;
\item $f_k(V_k')\subset V_k$\;;
\item\label{fourthpaper:fprop} $\|f_k-\id\|_{Y_{c_k}\setminus V_k'}<\epsilon_k$\;.
\end{enumerate}
\end{enumerate}
Observe that to apply Theorem \ref{fourthpaper:Erlend}, $V_k'$ needs to be a neighborhood of the locally exposable point $\zeta_k$; moreover to have convergence (see next section) it has to avoid the image of (the fixed compact $M$ and) the compact $C_{k-1}$. This cannot happen if $g$ is constant, as $F_{k-1}(g(M\cup C_{k-1}))=K_k=\{\zeta_k\}$ so $V_k'$ would be a punctured neighborhood of $\zeta_k$; moreover $\zeta_k\in Y_{c_k}\setminus V_k'$ and \ref{fourthpaper:fprop} would fail for $\epsilon_{k}$ small enough.
Set $K_{k+1}:=f_k(K_k)$; up to perturbing $f_k$ as we did with $g$, there exists $c_{k+1}\in\R$ such that $K_{k+1}\subseteq Y_{c_{k+1}}$ and $\exists\zeta_{k+1}\in K_{k+1}\cap\partial Y_{c_{k+1}}\setminus C$; then  $F_k:=f_k\circ F_{k-1}$ is holomorphic on some neighborhood of $K_1$ and $F_k(K_1)=K_{k+1}$; we finally choose a compact $C_k\subset X$ large so that $\max\{d(x,F_k(g(C_k))):x\in K_{k+1}\}\le\epsilon_k$ and $C_{k-1}\subsetneq C_k$. The induction may proceed.
\subsection{Construction of $f_k$ for Theorem \ref{fourthpaper:thm1}}\label{fourthpaper:indthm1}
\subsubsection{Existence of a locally exposable point}\label{fourthpaper:elep1}
Let $g\in \mathcal O_{\operatorname{e}}(X,Y)$.
Exploiting sharshapedness and up to considering $g_{\delta}(z)=g((1-\delta)z)$ for $0<\delta<1$, we can suppose without loss of generality $g$ to be holomorphic and injective on $U$ a Stein neighborhood of $\overline X$.
Call $R=\max_{z\in\overline X}|z|$, let $\zeta_0\in\partial X$ be such that $|\zeta_0|=R$ and define $\rho(z):=|z|^2-R^2$ which is strictly plurisubharmonic, hence $\zeta_0$ is a locally exposable point by Remark \ref{fourthpaper:boundaryptslocexp}.
Define $\zeta_1:=g(\zeta_0)$ and $K_1:=g(\overline{X})$.

If $\dim Y=n$, then $\zeta_1$ is locally exposable with respect to $\rho_1:=\rho\circ g^{-1}$ and $K_1$ is a Stein compact, in fact $W_{\alpha}=g_\alpha(X),\;\;0<\alpha<\delta$ is a neighbourhood basis of Stein domains since each $W_\alpha$ is biholomorphic to $X$ which is holomorphically convex. In this case, $\zeta_1$ is locally exposable and $K_1$ is a Stein compact asking $g$ for mere injectivity.

Let us now prove that $\zeta_1$ and $K_1$ are still a locally exposable point and a Stein compact respectively even in the case $\dim Y=m>n$.
Working with local charts we can assume to work on open subsets of $\C^m$.
Since $dg(\zeta)$ has full rank $n$ at each point $\zeta\in U$, it is $\imm dg(\zeta)\simeq \C^n$ and up to a linear change of coordinates, we can assume $\imm dg(\zeta_0)=\C^n\times\{0\}^{m-n}$ and obviously $\imm dg(\zeta_0)\bigoplus\spann_{\C}(e_{n+1},\dots,e_m)=\C^m$. Define then, for $(z_1,\dots,z_m)\in U\times\C^{m-n}$
$$
\tilde g(z_1,\dots,z_m):=g(z_1,\dots,z_n)+z_{n+1}e_{n+1}+\cdots+z_{m}e_{m}\;.
$$
Clearly $\tilde g(\zeta_0,0)=\zeta_1$ and it is locally invertible near $(\zeta_0,0)$. Call $h\colon A\to B$ the inverse, where $A,B\subset\C^m$ are open neighborhoods of $\zeta_1$ and $(\zeta_0,0)$ respectively and since
$$
\pi_j\circ h(z)=z_j\;\;\mbox{for}\;\;j=n+1,\dots,m\;,
$$
we have that
$$
g(U)\cap A=
\{z\in A\;:\;z_{n+1}=\cdots=z_m=0\}\;;
$$
we worked around $\zeta_1$, but the same argument holds for any other point (regardless of whether it is regular or singular) so $g(U)$ is a complex subvariety; in particular, it is \emph{locally closed} (every point in $g(U)$ has an open neighborhood $W$ such that $g(U)\cap W$ is closed in $W$), thus it admits a Stein neighborhood basis (\cite{fourthpaper:F17}, Theorem 3.1.1). The same holds for any $g(U')$, with $\overline{X}\subset U'\subset U$, therefore $K_1$ is a Stein compact. Define now
$$
\rho_1(z):=\rho\circ \alpha(z)+\sum_{j=n+1}^m|z_j|^2\;,
$$
where $\alpha:=\pi_{|\C^n}\circ h\colon A\to U$ and $\pi_{|\C^n}\colon\C^m\to\C^n$ is the projection on the first $n$ coordinates. The term $\rho\circ \alpha$ is plurisubharmonic as $\rho$ is such and $\alpha$ is holomorphic; moreover
$$
L_w(\rho\circ\alpha;t)=L_{\alpha(w)}(\rho;\alpha'(w)t)>0
$$
for every $t\in\C^m,\;\pi_{|\C^n}(t)\neq0,\;w\in A$,
in fact $\ker \alpha'(w)=\{0\}^n\times\C^{m-n}$ for any $w\in A$ and $\rho$ is strictly plurisubharmonic. Therefore we add the plurisubharmonic term $\beta(z):=\sum_{j=n+1}^m|z_j|^2$. Clearly
$$
L_w(\beta;t)=\beta(t)>0
$$
for every $t\in\C^{m},\;\pi_{|\C^{m-n}}(t)\neq0,\;w\in A$. Therefore $\rho_1$ is strictly plurisubharmonic on $\zeta_1$ (actually on the whole $A$). It is clear that $\rho_1(\zeta_1)=0$ and $\rho_1<0$ on $A\cap K_1\setminus\{\zeta_1\}$; then a computation shows that $d\rho_1(\zeta_1)\neq0$, hence $\zeta_1$ is locally exposable in $K_1$ with respect to $\rho_1$. 

\subsubsection{Inductive procedure}\label{fourthpaper:indpr1}
Let us observe that $Y\setminus K_1$ is connected, in fact if $n=1$, $X$ is simply connected (being starshaped) thus its boundary is connected and so is its image.

Let $n\ge2$ and $\dim Y=n$; then $g(X)$ is a Stein domain, which has connected boundary for $\dim Y\ge2$ (see \cite{fourthpaper:Ser55}, pag. 22).

The remaining case is $n\ge2$ and $m=\dim Y>n$. In this case just observe that $g(X)$ is a complex submanifold of complex codimension $m-n\ge1$, so its complement is connected. 

Assume we have $K_k\subset Y$ Stein compact, $\zeta_k\in K_k$ locally exposable with respect to some strictly plurisubharmonic $\rho_k$, $C_{k-1}\subset X$ compact, $F_{k-1}$ holomorphic injective on some neighborhood of $K_1$ such that $F_{k-1}(K_1)=K_{k}$. The construction of $f_k\colon U_k\to Y$ is as in Section \ref{fourthpaper:indpr2}, with $Y_{c_k}=K_k$ (Stein compact) and $\gamma_k((0,1])\subset Y\setminus Y_{c_k}$ (which can be achieved since $Y\setminus Y_{c_k}$ is connected, as above), allowing $f_k$ to be injective on $U_k$ and setting $\zeta_{k+1}:=f_k(\zeta_k)=q_{n_k}$ (which is locally exposable with respect to $\rho_{k+1}:=\rho_k\circ f_k^{-1}$).
Finally set $K_{k+1}:=f_k(K_k)$ Stein compact, $F_k:=f_k\circ F_{k-1}$ holomorphic injective on some open neighborhood of $K_1$ and $C_k\subset X$ as in Section \ref{fourthpaper:indpr2}. In particular $F_k(K_1)=K_{k+1}$ and $\{C_k\}_k$ is a compact exhaustion of $X$. 

\subsection{Construction of $f_k$ for Theorem \ref{fourthpaper:thm3}}
Since $K$ is a Stein compact, there exists $U\subset Y$ Stein neighborhood of $K$ and consequently a plurisubharmonic exhausting function $\rho\colon U\to\R$ as recalled in Theorem \ref{fourthpaper:morse}.
Then, the existence of at least one locally exposable point $x_0\in K$ is guaranteed by the argument of Section \ref{fourthpaper:elep2}, with $x_0, K, U$ playing the role of $\zeta_1, K_1$ and $Y$ respectively and being $\{C_k\}_k$ compact exhaustion for $K\setminus\{x_0\}$.
So, given any $x_0\in K$ locally exposable, the rest of the inductive procedure is as in Section \ref{fourthpaper:indpr1}, except for $M$ and $g$ which here just do not play any role, with $x_0, K, \rho$ playing the role of $\zeta_1, K_1$ and $\rho_1$ respectively.

\subsection{Construction of $f_k$ for Theorem \ref{fourthpaper:thm4}}
The construction is as in Section \ref{fourthpaper:indthm2}, with no $M$, no $g$, with $K$ playing the role of $K_1$ and $\{C_k\}_k$ compact exhaustion for $K\setminus\Gamma$.
At each step we get a locally exposable point $\zeta_{k}\in K_k\cap\partial Y_{c_k}\setminus C$, sent to $q_{n_k}$ by $f_k$ and corresponding to some $x_k\in K=K_1$, which is locally exposable as well (thus $\{x_k\}_k\subset\Gamma$).

\section{Proofs}

\subsection{Proof of Theorem \ref{fourthpaper:thm2}}
\begin{proof}
$F_k\colon K_1= \overline{g(X)}\to Y$ is holomorphic. Then from \ref{fourthpaper:ind}$_{+1}$ it follows that for every fixed $j$
\begin{align}\label{fourthpaper:claim}
F_k(g(C_j))\subset K_{k+1}\setminus V_{k+1}'
\end{align}
holds true for every $k\ge j$. 
Hence we get that, for every fixed $j$
\begin{equation}\label{fourthpaper:ineq}
\|F_{k+1}-F_k\|_{g(C_j)}=\|f_{k+1}-\id\|_{F_k(g(C_j))}\le\|f_{k+1}-\id\|_{K_{k+1}\setminus V_{k+1}'}<\epsilon_k\;
\end{equation}
is true for every $k\ge j$, so $\{F_k\}_k$ converges on compact subsets of $g(X)$ to $F\colon g(X)\to Y$ holomorphic. 
As above, \ref{fourthpaper:ind}$_{+1}$ implies $F_k(g(M))\subset K_{k+1}\setminus V_{k+1}' $ for every $k$, hence inequality $\eqref{fourthpaper:ineq}$ holds true for all $k$, thus
\begin{align*}
\|F-\id\|_{g(M)}
&\le\sum_{k\ge0}\|F_{k+1}-F_k\|_{g(M)}<\epsilon\;,
\end{align*}
allowing us to conclude that $\|h-g\|_M<\epsilon$, where $h=F\circ g\colon X\to Y$ is the approximating mapping.
We now check it actually has dense image.
If $h(X)$ is not dense in $Y$, then there exists an open ball $B\subset Y$ such that
$$
\beta:=d(B,h(X))>0\;.
$$
The construction of the sets $K_k$ and the sequence $\{n_k\}_{k\ge1}$ allows to consider a partition of $Q$ as
\begin{align*}
&q_{n_{k-1}},\dots,q_{n_{k}-1}\in{K_{k}},\;\;k\ge1\;.
\end{align*}
In this way we can define the sequence
$$
k(n):=j\;\;\mbox{for}\;\;n=n_{j-1},\dots,n_j-1,\;\;\mbox{for}\;\;j\ge1\;
$$
and we have
$q_n\in K_{k(n)}$ holds true for all $n$;
 the sequence $n\mapsto k(n)$ is increasing and such that $k(n)\to+\infty$ as $n\to+\infty$ (otherwise $\exists\widetilde k,N$ such that $q_n\in K_{\widetilde k}$ for all $n\ge N$, so $Q$ would not be dense). Since $g(X)\subset\subset Y$, it follows by Lemma \ref{fourthpaper:tech} that 
$$
d(q_n,h(X))= d(q_n,F(g(X)))\le\sum_{j\ge k(n)-1}\epsilon_{j}\;.
$$
This last sum is less than $\beta$ for any $n\ge n_{\beta}$, for a suitably large $n_\beta$. Therefore $\{q_n\}_{n>n_{\beta}}$, which is still dense in $Y$, does not meet an open ball, contradiction.
\end{proof}

\subsection{Proof of Theorem \ref{fourthpaper:thm1}}
\begin{proof} It is the same as the previous proof. Just observe that now $F_k$ is defined on $K_1=g(\overline X)$, it is holomorphic injective and so is $F$. Since $g$ is injective by assumption, then the approximating mapping $h=F\circ g$ is holomorphic injective as well.
\end{proof}

\subsection{Proof of Theorem \ref{fourthpaper:thm3}}
\begin{proof}
The mappings $F_k$ are defined, holomorphic and injective on some open neighborhood of $K$ and converge to $F\colon K\setminus\{x_0\}\to Y$ uniformly on compacts of $K\setminus\{x_0\}$. The construction of mappings $f_k$ ensures, as for Theorem \ref{fourthpaper:thm2}, to achieve $\overline {F(K\setminus\{x_0\})}=Y$.
\end{proof}

\subsection{Proof of Theorem \ref{fourthpaper:thm4}}
\begin{proof}
As for Theorem \ref{fourthpaper:thm3} (except for injectivity of $F_k$), with $K\setminus\Gamma$ instead of $K\setminus\{x_0\}$.
\end{proof}

\end{document}